\documentclass[a4paper,10pt]{article}

\usepackage{amsmath,amsthm,amsfonts,amssymb,stmaryrd,mathrsfs}
\usepackage{latexsym}
\usepackage{fullpage}
\usepackage{a4,color,palatino,fancyhdr}
\usepackage{graphicx}
\usepackage{float}  
\usepackage[small, bf, margin=90pt, tableposition=bottom]{caption}
\RequirePackage[T1]{fontenc}
\usepackage[breaklinks, naturalnames]{hyperref}
\usepackage[backrefs]{amsrefs}
\usepackage{enumerate}
\usepackage{multicol}

\input{xy}
\xyoption{all}
\xyoption{poly}
\usepackage[all]{xy}
\usepackage{colonequals}

\newtheorem{theorem}{Theorem}[section]
\newtheorem{proposition}[theorem]{Proposition}

\newtheorem{lemma}[theorem]{Lemma}

\numberwithin{equation}{section}

\def\cl#1{{\langle #1 \rangle}}

\def\Ker{\mathop{\rm Ker}\nolimits}
\def\im{\mathop{\rm Im}\nolimits}

\def\Hom{{\mathrm {Hom}}}

\newcommand\ZZ{{\mathbb{Z}}}
\newcommand\NN{{\mathbb{N}}}
\newcommand\KK{\mathbb{K}}

\title{Hochschild homology and cohomology of down-up algebras}
\author{Sergio Chouhy, Estanislao Herscovich, and Andrea Solotar
\thanks{{\footnotesize This work has been supported by the projects  
UBACYT 20020130100533BA, UBACYT 20020130200169BA, PIP-CONICET 11220150100483CO, and MATHAMSUD-REPHOMOL.
The second and third authors are research members of CONICET (Argentina).}}
}
\date{}

\begin{document}

\maketitle

\begin{abstract}
 {
We present a detailed computation of the cyclic and the Hochschild homology and 
cohomology of generic and $3$-Calabi-Yau homogeneous down-up algebras. 
This family was defined by Benkart and Roby in \cite{BR98} in their study of 
differential posets. 
Our calculations are completely explicit, by making use of the Koszul bimodule resolution 
and arguments similar to those appearing in \cite{HS12}.
}
\end{abstract}

\textbf{2010 Mathematics Subject Classification: 16E40, 16E05, 16W50.}

\textbf{Keywords:} Down-up algebra, homology, Hochschild, resolution.


\section{Introduction}

Motivated by the study of the algebra generated by the up and down 
operators in the theory of \emph{differential} posets defined independently by R. Stanley 
in \cite{Sta88} and by S. Fomin in \cite{Fo86}, or of \emph{uniform} posets defined by P. 
Terwilliger in \cite{Ter90}, G. Benkart and T. Roby introduced in \cite{BR98} the notion 
of \emph{down-up algebras}.  
They have been intensively studied in \cite{B}, \cite{BW01}, \cite{CM00}, \cite{Ku01}, 
\cite{Zh99} among many other articles, 
and different kinds of generalizations have been defined \cite{CL09}, \cite{CS04}.
Since the homological invariants of an algebra provide useful tools for its description as well as for its representations, 
many of their homological properties were studied 
and in particular, a quite convenient projective resolution of the regular bimodule of a 
down-up algebra was constructed by S. Chouhy and A. Solotar in \cite{CS15}. 

Let $\KK$ be a fixed field of characteristic $0$.
Given parameters $(\alpha,\beta,\gamma) \in \KK^{3}$, the associated 
\emph{down-up algebra} $A(\alpha,\beta,\gamma)$ is defined as the quotient 
of the free associative algebra $\KK\cl{u,d}$ by the ideal generated by the relations
\begin{equation}
\label{eq:relationsdownup}
\begin{split}
d^{2} u - (\alpha d u d + \beta u d^{2} + \gamma d),
\\
d u^{2} - (\alpha u d u + \beta u^{2} d + \gamma u).
\end{split}
\end{equation}
We shall sometimes denote a particular down-up algebra $A(\alpha,\beta,\gamma)$ just by 
$A$ 
to simplify the notation.

As examples of down-up algebras, $A(2,-1,0)$ is isomorphic to the enveloping 
algebra of 
the Heisenberg-Lie algebra of dimension $3$, and, for $\gamma \neq 0$, $A(2,-1,\gamma)$ 
is 
isomorphic to the enveloping algebra of $\mathfrak{sl}(2,\KK)$.  
Moreover, Benkart proved in \cite{B} that any down-up algebra such that $(\alpha, \beta) 
\neq (0,0)$ is isomorphic to one of Witten's $7$-parameter deformations of 
$\mathscr{U}(sl(2,\KK))$. 

Any of these algebras has a PBW basis given by 
\begin{equation}
\label{eq:pbwbasis}
     \{ u^{i} (d u)^{j} d^{k} : i, j, k \in \NN_{\geq 0} \}.     
\end{equation}

Note that the down-up algebra $A(\alpha,\beta,\gamma)$ can be regarded as a $\ZZ$-graded 
algebra where the degrees of $u$ and $d$ are respectively $1$ and $-1$. 
We shall refer to this grading as \emph{special}, and denote the special degree of an 
element $a \in A$ 
by $\mathrm{s\text{-}deg}(a)$. 
In fact, $A=\bigoplus_{n\in \ZZ}A_n$ where $A_n$ is the $\KK$-vector space spanned by the 
set $\{u^i(du)^jd^k| i-k=n\}$.

It is known \cite{CM00} that if $A(\alpha,\beta,\gamma)$ is isomorphic to 
$A(\alpha',\beta',\gamma')$, then 
\begin{equation}
\label{eq:isoproblem}
\begin{split}
&\text{both } \alpha + \beta \text{ and } \alpha' + \beta' \text{ are $1$ or different 
from $1$}, 
\\
&\text{both } \gamma \text{ and } \gamma' \text{ are $0$ or different from $0$}. 
\end{split}
\end{equation}
The down-up algebra $A(\alpha,\beta,\gamma)$ is isomorphic to $A(\alpha,\beta,1)$ for all 
$\gamma \neq 0$. 
Furthermore, if $\KK$ is algebraically closed, P. Carvalho and I. Musson showed in 
\cite{CM00} that $A(\alpha,\beta,\gamma)$ is isomorphic to $A(\alpha',\beta',\gamma')$ if 
and only if the following conditions hold
\begin{equation}
\label{eq:isoproblemdos}
\begin{split}
&\text{either } \alpha' = \alpha \text{ and } \beta' = \beta, \text{ or } 
\alpha' = - \alpha^{-1} \beta \text{ and } \beta' = \beta^{-1}, 
\\
&\text{both } \gamma \text{ and } \gamma' \text{ are $0$ or different from $0$}. 
\end{split}
\end{equation}

E. Kirkman, I. Musson and D. Passman proved in \cite{KMP99} that $A(\alpha,\beta,\gamma)$ 
is noetherian if and only it is a domain, which is tantamount to the fact that the 
subalgebra of $A(\alpha,\beta,\gamma)$ generated by $u d$ and $d u$ is a polynomial 
algebra in two indeterminates, that in turn is equivalent to $\beta \neq 0$. 
Under either of the previous situations, $A(\alpha,\beta,\gamma)$ is Auslander regular 
and 
its global dimension is $3$.
On the other hand, it was proved by Cassidy and Shelton in \cite{CS04} than, if $\KK$ is 
algebraically closed, then the global dimension of $A(\alpha,\beta,\gamma)$ is always 
$3$. 
Moreover, Benkart and Roby proved in \cite{BR98} that the Gelfand-Kirillov dimension of a 
down-up algebra is $3$, independently of the parameters.
Since $A(\alpha,\beta,\gamma)$ is isomorphic to the opposite algebra, left and right 
dimensions coincide.

The centre of a down-up algebra has been computed in \cite{Ku01} and \cite{Zh99}, and the 
first Hochschild cohomology space of a localization of some families of down-up algebras 
with $\gamma=0$ has been recently computed in \cite{Ta15}, but up to now there is no 
description of Hochschild homology and cohomology of down-up algebras available.

The main result of this article is the computation of the complete Hochschild homology 
and cohomology of two families of down-up algebras with $\gamma=0$. Given 
$\alpha,\beta\in\KK$, denote $r_1$ and $r_2$ the roots of the polynomial $t^2 - \alpha t 
-\beta$.
We define the following two cases.
\begin{itemize}
 \item[(F1)] \textit{Graded generic down-up algebras}.
The algebra $A(\alpha,\beta,0)$ belongs to this family if and only if 
$(\alpha,\beta)\neq(0,0)$ and  $r_1^ir_2^j\neq1$  for all $i$ and $ j$ such that 
$(i,j)\neq(0,0)$. We call this assumption the \textit{genericity 
hypothesis}.
 \item[(F2)] \textit{Graded $3$-Calabi-Yau down-up algebras}. The algebra 
$A(\alpha,\beta,0)$ belongs to this family if and only if $\beta=-1$,
in which case $r_2=r_1^{-1}$.
\end{itemize}
The methods we use are closely related to those used for the computation of the Hochschild 
and cyclic 
(co)homology of Yang-Mills algebras in \cite{HS12}, and we think that they will lead to 
the 
computation of these invariants for the other cases as well, with more 
involved calculations.
We are not studying here the case $A(0,0,0)$ for which the resolution constructed by M. Bardzell in \cite{Ba97} 
is available.

In Section \ref{sec:main} we introduce some notations and basic objects such as the 
projective resolution  of $A$ as $A$-bimodule. In case $\gamma=0$, this is the Koszul 
resolution. We state the main results of the article in Theorem \ref{thm:homology} and 
Theorem \ref{thm:cohomology}, and leave the proofs for the subsequent sections. 

In Section \ref{sec:homology} we compute Hochschild and cyclic homology. It is clear from 
the resolution that $HH_i(A)=0$ for all $i\geq4$. We provide explicit bases for $HH_0(A)$ 
and $HH_3(A)$ and we use a Hilbert series argument involving reduced cyclic homology and 
a theorem by K. Igusa in \cite{Ig92} to obtain the Hilbert series of $HH_1(A)$ and $HH_2(A)$.

Section \ref{sec:cohomology} is devoted to the Hochschild cohomology. Since $A(\alpha,-1,0)$ 
is $3$-Calabi-Yau, we only study here algebras belonging to the family (F1). 
It is well known that their centre is $\KK$ (see \cites{Ku01,Zh99}). 
We give bases of $HH^1(A)$, $HH^2(A)$ and $HH^3(A)$. 
This may be particularly useful for the description of the corresponding deformations.

\section{Main results}
\label{sec:main}
In this section we will introduce some elements of down-up algebras with the aim of 
stating the main results of the article, that will be proved in the sequel.

As stated previously, we will usually denote $A(\alpha,\beta,\gamma)$ simply by $A$. 
We mentioned in the introduction that this algebra can be regarded as a $\ZZ$-graded 
algebra where the degrees of $u$ and $d$ 
are, $1$ and $-1$, respectively. 
We shall refer to this grading as \emph{special}. 
If $\gamma$ is zero, the algebra $A$ has another grading that we will call \textit{usual} 
with $u$ and $d$ both in 
degree $1$. 
We shall denote the usual degree of an element $a \in A$ by $\mathrm{deg}(a)$. 
Notice that the homogeneous components with respect to the usual grading are finite 
dimensional $\KK$-vector spaces. 
For $\gamma=0$, $A$ is thus $\ZZ^2$-graded by $\mathrm{bideg} \colonequals 
(\mathrm{deg},\mathrm{s\text{-}deg})$.

Let $V$ be the $\KK$-vector space spanned by the set $\{d,u\}$ and let 
$T(V)=\oplus_{n\geq0}V^{\otimes n}$ be the tensor algebra
of $V$ over $\KK$. 
We will typically omit the tensor product symbols when writing an 
element of $T(V)$. 
Let $R$ be the subspace of $V^{\otimes3}$ spanned by the set $\{d^2u,du^2\}$ and let 
$\Omega$ be the subspace of $V^{\otimes4}$ spanned by the element $d^2u^2$.

There is a short projective resolution of $A$ as $A$-bimodule (see \cite{CS15}):
\begin{align}
\label{resolucion}
     0 \rightarrow A \otimes \Omega \otimes A 
           \overset{\delta_{3}}{\rightarrow} A \otimes R \otimes A 
           \overset{\delta_{2}}{\rightarrow} A \otimes V \otimes A 
           \overset{\delta_{1}}{\rightarrow} A \otimes A 
           \rightarrow 0,
\end{align}
where the augmentation $\delta_{0}:A\otimes A\to A$ is given by the multiplication map. 
The differentials are
\begin{equation}
\label{eq:delta2}
\begin{split}
\delta_{1}(1\otimes v\otimes 1) &=  v \otimes 1 - 1 \otimes v, \hskip0.5cm \text{for all 
}v\in V,\\
\delta_{2}(1 \otimes d^{2} u \otimes 1) &= 1 \otimes d \otimes d u +  d \otimes d 
\otimes u 
+ d^{2} \otimes u \otimes 1 
\\
&- \alpha (1 \otimes d \otimes u d + d \otimes u \otimes d  + d u \otimes d 
\otimes 1)
\\
&- \beta (1\otimes u \otimes d^{2} + u \otimes d \otimes d  + u d \otimes d 
\otimes 1)
\\
&- \gamma \otimes d \otimes 1,
\end{split}
\end{equation}
\begin{equation}
\label{eq:delta2bis}
\begin{split}
\delta_{2}(1 \otimes d u^{2} \otimes 1) &= 1 \otimes d \otimes u^{2}  + d \otimes u 
\otimes u  
+ d u \otimes u \otimes 1 
\\
&- \alpha (1\otimes u \otimes d u  +  u \otimes d \otimes u  +  u d \otimes u 
\otimes 1)
\\
&- \beta (1\otimes u \otimes u d  + u \otimes u \otimes d  +  u^{2} \otimes d 
\otimes 1)
\\
&- \gamma  \otimes u \otimes 1,
\end{split}
\end{equation}
and 
\begin{equation}
\label{eq:delta3}
\begin{split}
\delta_{3}(1 \otimes d^{2} u^{2} \otimes 1) &= d \otimes d u^{2} \otimes 1 + \beta 
\otimes d u^{2} \otimes d  
\\
&-1 \otimes d^{2} u \otimes u  - \beta u \otimes d^{2} u \otimes 1.
\end{split}
\end{equation} 
Notice that for all $i$ the map $\delta_i$ is homogeneous for the special 
degree and the same holds for the usual degree when $\gamma=0$. 
Moreover, the projective resolution \eqref{resolucion} is minimal in the category of 
graded modules if $\gamma=0$.

We will denote the complex  \eqref{resolucion} by $K$. It is not hard to see that if 
$\beta\neq0$, then there is an isomorphism $\Hom_{A^e}(K,A^e) \cong 
A_\sigma\otimes_A K$ of complexes of $A$-bimodules, where 
$A^e=A\otimes A^{op}$ and $\sigma$ is the automorphism determined by 
$\sigma(u)=-\beta^{-1}u$ and 
$\sigma(d) =-\beta d$. As a consequence, noetherian down-up algebras are twisted 
$3$-Calabi-Yau, and if $\beta=-1$ they are $3$-Calabi-Yau. Moreover, the algebra 
$A(\alpha,\beta,\gamma)$ is $3$-Calabi-Yau if and only if $\beta=-1$. See for example 
\cite{SL16} and the references therein.

In the following theorem we summarize the results about Hochschild homology.
\begin{theorem}
\label{thm:homology}
 Let $A=A(\alpha,\beta,0)$ be a down-up algebra. Define
\begin{align*}
 &s_1 = \frac{t (2+3 t)}{1-t^2},
 &
 &s_2 = \frac{t^2}{1-t^4},
\end{align*}
and for $n\geq1$,
\begin{align*}
 &f_n = \frac{1}{(1-t^4)(1-t^n)^2},
 &
 &g_n = \frac{t^2-t^{2n}}{1-t^4},
 &
 &h_n = \frac{2t(1 - t^{n-1})}{(1-t)(1-t^n)}.
 \end{align*}
 
\begin{itemize}
 \item If $A$ belongs to (F1), then the Hilbert series of the Hochschild homology spaces 
$HH_0(A)$ and $HH_1(A)$ are the following, while $HH_2(A)$ and $HH_3(A)$ vanish.
\begin{align*}
 &HH_0(A)(t)=\frac{1+2t+2t^2}{1-t^2},
 &
 &HH_1(A)(t)=\frac{t(2+3t)}{1-t^2},
 \end{align*}	
 \item If $A$ belongs to the family (F2) and $r_1$ is not a root of unity, then the 
Hilbert series are as follows.
\begin{align*}
 &HH_0(A)(t)=\frac{1+2t+2t^2}{1-t^2},
 &
 &HH_1(A)(t)= \frac{t (2+3 t)}{1-t^2} + \frac{t^4}{1-t^8},
 \\
 &HH_2(A)(t)= \frac{2t^4}{1-t^8},
 &
 &HH_3(A)(t)=\frac{t^4}{1-t^8}.
\end{align*}

\item If $A$ belongs to (F2) and $r_1$ is a primitive $n$-th root of unity, then
\begin{enumerate}[i)]
 \item For $n$ even and $n\geq4$,
 \begin{align*}
 &HH_0(A)(t)= f_n + h_n + s_2,\\
 &HH_1(A)(t)= \frac{t^4}{(1-t^4)(1-t^{n})^2} + 2(f_n + h_n + s_2-1)- s_1,\\
 &HH_2(A)(t)=\frac{2t^4}{(1-t^4)(1-t^{n})^2}+f_n + h_n + s_2-s_1 -1, \\
 &HH_3(A)= \frac{t^4}{(1-t^4)(1-t^{n})^2}.
 \end{align*}
 \item For $n$ odd and $n\geq3$,
 \begin{align*}
  &HH_0(A)(t) = f_n + g_n + h_n,\\
  &HH_1(A)(t) = \frac{t^4}{(1-t^4)(1-t^{n})^2} + 2(f_n + g_n + h_n-1)-s_1, \\
  &HH_2(A)(t) = \frac{2t^4}{(1-t^4)(1-t^{n})^2} + f_n + g_n + h_n-s_1-1,\\
  &HH_3(A)(t) = \frac{t^4}{(1-t^4)(1-t^{n})^2}.
 \end{align*}
 \item For $n=2$, that is $r_1=-1$, the Hilbert series of the Hochschild homology spaces 
 of $A(2,-1,0)$ are
 \begin{align*}
 &HH_0(A)(t)=\frac{1+2t+2t^2 - t^4 - 2t^5}{(1-t^2)^2(1+t^2)},
 &
 &HH_1(A)(t)=\frac{2t+3t^2+t^4-2t^5}{(1-t^2)^2(1+t^2)},\\
 &HH_2(A)(t)=\frac{2t^4}{(1-t^2)^2(1+t^2)},
 &
 &HH_3(A)(t)=\frac{t^4}{1-t^4}.
 \end{align*}
 \item For $n=1$, that is $r_1=1$, the Hilbert series of the Hochschild homology spaces 
 of $A(2,1,0)$ are
 \begin{align*}
 &HH_0(A)(t)=\frac{1}{(1-t)^2},
 &
 &HH_1(A)(t)=\frac{t(2-t)(1+t^2)}{(1-t)^2},\\
 &HH_2(A)(t)=\frac{2t^3(1+t-t^2)}{(1-t^2)(1-t)},
 &
 &HH_3(A)(t)=\frac{t^4}{1-t^2}.
\end{align*}
\end{enumerate}
\end{itemize}
\end{theorem}

While proving this result we will also obtain the Hilbert series of the cyclic homology 
of $A$. Moreover, we give explicit bases of $HH_0(A)$ and $HH_3(A)$.

The computation of the Hilbert series of Hochschild cohomology spaces follows from the 
previous ones in the $3$-Calabi-Yau case, that is, for the family (F2). 
However, we want to describe what happens for an algebra $A$ in (F1). 
No formula involving cyclic homology and the respective Hilbert series is available in 
this context. So, we provide explicit basis for the Hochschild cohomology 
spaces in this case.

\begin{theorem}[see Section \ref{sec:cohomology} for the notation]
\label{thm:cohomology}
Let $A$ be a down-up algebra belonging to the family (F1).
The Hilbert series of the Hochschild cohomology spaces are as follows.
\begin{align*}
&HH^0(A)(t) = 1,
&
&HH^1(A)(t) = 2,\\
&HH^2(A)(t) = \frac{1}{t^2} + 2 + \frac{t^2}{1-t^2},
&
&HH^3(A)(t) = \frac{1}{t^4(1-t^2)}.
\end{align*}
Moreover, 
\begin{enumerate}[i)]
 \item $HH^0(A)=\KK$,
 \item the classes of the elements $D|d$ and $U|u$ form a basis of $HH^1(A)$,
 \item the classes of the elements $\{D^2U|w_1^kd + DU^2|uw_1^k: 
k\geq0\}\cup\{D^2U|ud^2 + DU^2|u^2d\}$ form a basis of $HH^2(A)$, and
 \item the classes of the elements $\{D^2U^2|w_1^j:j\geq0,\text{ and 
}j\neq2\}\cup\{D^2U^2| uw_1d\}$ form a basis of $HH^3(A)$.
\end{enumerate}

\end{theorem}

From the previous results we remark that all Hochschild homology spaces are either infinite dimensional -- with finite dimensional graded components-- or zero.
The situation differs for Hochschild cohomology of algebras belonging to the family (F1), in which case the centre is as small as possible, 
that is one dimensional, and the first cohomology space has dimension $2$, containing just the two obvious derivations. 
The fact that the second and third cohomology spaces are infinite dimensional in all cases suggests that deformations of down-up algebras 
are quite complicated, but having explicit bases in case (F1) indicates that the deformation theory may be nonetheless tractable.

\section{Hochschild and cyclic homology of down-up algebras}
\label{sec:homology}

From now on we fix $\gamma=0$, and let $A=A(\alpha,\beta,0)$ be a down-up algebra 
with $(\alpha,\beta)$ different from $(0,0)$. 
In this section we assume that the field $\KK$ contains both roots of the polynomial 
$t^2-\alpha t-\beta$ and it is of characteristic 
zero.

Denote by $A(t,s)$ the Hilbert series of the bigraded algebra $A$.
Consider $\KK$ as a left $A$-module with the trivial action of $d$ and 
$u$. 
By computing the Euler-Poincar\'e characteristic of the exact complex 
$K\otimes_A\KK$ we obtain
\begin{equation}
\label{eq:HilbSerAtotal}
     A(t,s) = \frac{1}{1- t (s +s^{-1}) + t^{3} (s +s^{-1}) - t^{4}}.
\end{equation}
The Hilbert series for the usual grading is obtained by setting $s=1$ in the previous 
expression:
\begin{equation}
\label{eq:HilbSerA}
     A(t) = \frac{1}{(1-t^{2})(1-t)^{2}}.
\end{equation}

Next we describe a basis of $A$ as a $\KK$-vector space that will be useful for the computations. 
Denote $r_1$ and $r_2$ the roots of the polynomial $t^2 - \alpha t - \beta$. 
Since $(\alpha,\beta)\neq (0,0)$ we may assume that $r_1$ is not zero.
For $l\in\{1,2\}$ we define $w_l = \beta ud + r_l du$. It is straightforward that
\begin{equation}
\label{eq:permws}
\begin{split}
w_{l} u &= r_{l} u w_{l},
\\
d w_{l} &= r_{l} w_{l} d,
\end{split}
\end{equation}
for $l=1,2$. Given $p \in \ZZ_{\geq 0}$, denote 
\[     
  \phi_{p} = \sum_{i=0}^{p} r_1^{i} r_2^{p-i} = \frac{r_1^{p+1}-r_2^{p+1}}{r_1-r_2}.  
\]
The last expression only holds for $r_1\neq r_2$. We set $\phi_{-1} = 0$. 
The following identities are easily proved by induction. 
\begin{lemma}
\label{fact:comm}
For all $k\geq0$ the following equalities hold
\begin{equation}
\begin{split}
   d u^{k} &= \frac{\phi_{k-1}}{r_{1}} u^{k-1} w_{1}  + r_{2}^{k} u^{k} d,     
\\
   d^{k} u &= \frac{\phi_{k-1}}{r_{1}} w_{1} d^{k-1} + r_{2}^{k} u d^{k}.
\end{split}
\end{equation}
\end{lemma}

For the proof of the following result we refer to \cite{Zh99}, Lemma 2.2. 
\begin{lemma}
Let $l\in\{1, 2\}$ and suppose $r_l$ is not zero. The set $\{ u^{i} w_{l}^{j} d^{k} : i, 
j, k \in \NN_{\geq 0} \}$ 
is a basis of $A$.
\end{lemma}

We denote $\overline{HC}_\bullet(A)$, $\overline{HH}_\bullet(A)$ 
and $\overline{HH}^\bullet(A)$ the reduced cyclic homology, the reduced Hochschild homology and 
the reduced Hochschild cohomology of $A$. Notice that the reduced Hochschild homology and cohomology spaces differ from the non reduced groups only in (co)homological degree zero.

\bigskip

Tensoring the resolution $K$ of $A$ given by \eqref{resolucion} by $A$ over $A^{e}$ we obtain the following complex, whose homology is 
isomorphic to the Hochschild homology of $A$:
\[     0 \rightarrow A \otimes \Omega
           \overset{d_{3}}{\rightarrow} A \otimes R 
           \overset{d_{2}}{\rightarrow} A \otimes V  
           \overset{d_{1}}{\rightarrow} A 
           \rightarrow 0,     \]
where $d_{1}(a \otimes d + a' \otimes u) = a d - d a + a' u - u a'$, 
\begin{equation}
\label{eq:d2}
\begin{split}
&d_{2}(a \otimes d^{2} u + a' \otimes d u^{2}) 
\\
&= \Big(d u a + u a d + u^{2} a' - \alpha (u d a + a d u + u a' u) 
- \beta (d a u + a u d + a' u^{2}) - \gamma a \Big) \otimes d 
\\
&+ \Big(a d^{2} + u a' d + a' d u - \alpha (d a d + d u a' + a' u d) 
- \beta (d^{2} a + u d a' + d a' u) - \gamma a' \Big) \otimes u,
\end{split}
\end{equation}
and 
\begin{equation}
\label{eq:d3}
d_{3}(a \otimes d^{2} u^{2}) = - (u a + \beta a u) \otimes d^{2} u + (a d + \beta d a) \otimes d u^{2}.
\end{equation}

Since the characteristic of $\KK$ is zero, a theorem by T. Goodwillie
(see \cite{Go85}, and the consequence indicated by M. Vigu\'e-Poirrier in \cite{VP91}) tells us that,  
for all $i \in \NN_{\geq 0}$ there are short exact sequences of graded 
vector spaces
\[     0 \rightarrow \overline{HC}_{i-1}(A) \rightarrow \overline{HH}_{i}(A) \rightarrow \overline{HC}_{i}(A) \rightarrow 0.     \]
Since $HH_i(A)=0$ for all $i\geq4$, we deduce that $\overline{HC}_i(A)=0$ for all 
$i\geq3$.

We recall that the the Euler-Poincar\'e characteristic of the reduced cyclic homology 
$\chi_{\overline{HC}_{\bullet}(A)}(t)$ of $A$ is defined as
\[     \chi_{\overline{HC}_{\bullet}(A)}(t) = \sum_{p \in \ZZ} (-1)^{p} 
\overline{HC}_{p}(A)(t) = \overline{HC_0}(A)(t) - \overline{HC_1}(A)(t) + 
\overline{HC_2}(A)(t).     \]
A result by K. Igusa (see \cite{Ig92}, Thm. 3.5, Equation (16)) 
--and proved by different methods in an unpublished work by C. L\"ofwall-- tells us that it satisfies the identity 
$     \chi_{\overline{HC}_{\bullet}(A)}(t) = \sum_{\ell \in \NN} 
\frac{\varphi(\ell)}{\ell} \log (A(t^{\ell})),     $
where $\varphi$ is the Euler's totient function. 
Using that $\sum_{d|n}\varphi(d)=n$ for 
all $n\in\NN$, we compute
\begin{align*}
\sum_{\ell\in\NN}\frac{\varphi(\ell)}{\ell}\log(1-t^\ell)=-\sum_{n\in\NN}(\sum_{d|n}\varphi(d))\frac{t^n}{n}
=-\frac{t}{1-t}.
\end{align*}
Hence, using Equation \ref{eq:HilbSerA} we obtain
\begin{align*}
\chi_{\overline{HC}_{\bullet}(A)}(t) = - \sum_{\ell \in \NN}\frac{\varphi(\ell)}{\ell} 
\log 
((1-t^{2 \ell})(1-t^{\ell})^{2}) = \frac{t (2+3 t)}{1-t^2}.
\end{align*}
Putting this all together we get
\begin{align}
\label{eq:hilbert_series}
\nonumber&\overline{HC}_0(A)(t) = \overline{HH}_0(A)(t)\\
\nonumber&\overline{HC}_1(A)(t) = \overline{HH}_0(A)(t) + HH_3(A)(t) - 
\frac{t(2+3t)}{1-t^2},\\
&\overline{HC}_2(A)(t) = HH_3(A)(t),\\
\nonumber&HH_1(A)(t) = 2\overline{HH}_0(A)(t) + HH_3(A)(t) - \frac{t(2+3t)}{1-t^2},\\
\nonumber&HH_2(A)(t) = \overline{HH}_0(A)(t) + 2HH_3(A)(t) - \frac{t(2+3t)}{1-t^2}.
\end{align}
The computation of $HH_0(A)$ and $HH_3(A)$ will thus provide us the dimensions of the 
graded components of the other spaces. 

\begin{proposition}
\label{proposition:hh0}
Let $A(\alpha,\beta,0)$ be a down-up algebra.
\begin{enumerate}
\item If $A$ belongs to (F1), then the vector space 
$HH_{0}(A)$ has a basis formed by the classes of the elements of the set 
\[\{ 1, w_{1}^{j}, d^{j}, u^{j} : j \in \NN \}.\]
\item If $A$ belongs to $(F2)$, define $n$ as the order of $r_1$ if it is a root of 
unity and  $0$ otherwise. The vector space $HH_0(A)$ has a basis formed by the 
classes of 
the following elements.

\begin{itemize}
 \item If $n$ is even and different from $2$,
 \begin{enumerate}[i)]
 \item $u^iw_1^jd^k$ such that $n$ divides $j-i$ and $j-k$,
 \item $u^i$, $d^k$ with $i,k\geq0$ such that $n\nmid i$ and $n\nmid k$, and
 \item $w_1^{j}$, where $j$ is any odd number.
 \end{enumerate}
 \item If $n$ is odd and different from $1$,
 \begin{enumerate}[i)]
 \item $u^iw_1^jd^k$ such that $n$ divides $j-i$ and $j-k$.
 \item $u^i$, $d^k$ with $i,k\geq0$ such that $n\nmid i$ and $n\nmid k$, and
 \item $w_1^{j}$, where $j$ is odd and $j\leq n-2$.
 \end{enumerate}
 
 \item If $n=2$,
 \begin{enumerate}[i)]
  \item[] $u^{2i}d^{2k}, u^{2i+1},d^{2k+1},w_1^{2j+1}$ with $i,j,k\geq0$.
 \end{enumerate}
\item If $n=1$,
 \begin{enumerate}[i)]
  \item[] $u^id^k$ with $i,j\geq0$.
 \end{enumerate}
\end{itemize}
\end{enumerate}
\end{proposition}

Before we get to the proof of Proposition \ref{proposition:hh0} we need some definitions 
and auxiliary results.
Let $a = u^{i} w_{1}^{j} d^{k}$, where $i, j, k \in \NN_{\geq 0}$. Using Lemma 
\ref{fact:comm} we deduce that
\begin{equation}
\label{eq:1sthh0}     
 ad-da=(1 - r_{1}^{j} r_{2}^{i}) u^{i} w_{1}^{j} d^{k+1} - \frac{\phi_{i-1} }{r_{1}} 
u^{i-1} 
w_{1}^{j+1} d^{k},     
\end{equation}
and 
\begin{equation}
\label{eq:2ndhh0} 
    au-ua=  - (1 - r_{1}^{j} r_{2}^{k}) u^{i+1} w_{1}^{j} d^{k} + \frac{\phi_{k-1} 
}{r_{1}} 
u^{i} w_{1}^{j+1} d^{k-1}.     
\end{equation}
Define $f_{i-1,j+1,k} = ad-da$, and $g_{i,j+1,k-1}=ua-au$. Observe that $\im(d_1)$ is 
equal to the vector space spanned by the set 
\[
 \{f_{i,j,k}:i\geq-1,j\geq1,k\geq0\}\cup\{g_{i,j,k}:i\geq0,j\geq1,k\geq-1\}.
\]
Let us write $t_{i} = \phi_{i}/r_1$ and $s_{i,j}=1-r_1^jr_2^i$. Then 
$f_{i,j,k}=s_{i+1,j-1}u^{i+1}w_1^{j-1}d^{j+1} - t_{i}u^iw_1^jd^k$.

For $i,j,k$ with $t_{i}\neq0$ and $j\geq1$, let 
\[L_{i,j}\colonequals\max\{l: \text{ such that }0\leq l\leq j-1\text{ and 
}t_{i+l}\neq0\}\] 
and
\[
z_{i,j,k} \colonequals -\frac{1}{t_i}f_{i,j,k} - \sum_{l=1}^{L_{i,j}} 
\left(\frac{1}{t_{i}}\prod_{m=1}^{l}\frac{s_{i+m,j-m}}{t_{i+m}}\right)f_{i+l,j-l
, k+l },
\]
where we omit the second summand whenever $L_{i,j}=0$. In order to simplify notations, let
$L=L_{i,j}$. Notice that 
\[
 z_{i,j,k} = u^iw_1^jd^k - 
\left(\frac{s_{i+L+1,j-L-1}}{t_{i}}\prod_{m=1}^{L}\frac{s_{i+m,j-m}}{t_{i+m}}\right)u^{
i+L+1 } w_1^ { j-L-1 } d^ { k+L+1 } , 
\]
and that it belongs to $\im(d_1)$.
On the other hand, define
\[
\Gamma=\{(i,j,k)\in\NN_0^3: r_1^jr_2^i=1 \text{ or 
}k=0\}\cap\{(i,j,k)\in\NN_0^3:r_1^jr_2^k=1\text{ or 
}i=0\}.
\]
\begin{lemma}
\label{lemma:hh0_1}
 Let $i,j,k\geq0$ and let $x\in\im(d_1)$ be such that the coefficient of $u^iw_1^jd^k$ in 
$x$ is not zero. If $(i,j,k)\in\Gamma$, then $j\geq1$. If in addition $n|i-k$, then 
$t_i\neq0$. 
\end{lemma}
\begin{proof}
 Write $x= \sum_{a,b,c\geq0}(\epsilon_{a,b,c}[u^aw_1^bd^c,d] + 
\mu_{a,b,c}[u^aw_1^bd^c,u])$. The coefficient of $u^iw_1^jd^k$ in this expression is
\[
 \epsilon_{i,j,k-1}s_{i,j} - \epsilon_{i+1,j-1,k}t_{i} - 
\mu_{i-1,j,k}s_{k,j} + \mu_{i,j-1,k+1}t_{k},
\]
where elements with negative subindices are zero. If $(i,j,k)$ belongs 
to $\Gamma$, then this element is equal to $-\epsilon_{i+1,j-1,k}t_{i} + \mu_{i,j-1,k+1}t_{k}$. 
By hypothesis this is not zero. We deduce that $j\geq1$. If $n|i-k$, then $t_{k}=t_{i}$ 
and the last expression is equal to $(-\epsilon_{i+1,j-1,k} + 
\mu_{i,j-1,k+1})t_{i}$. Since this is not zero, we obtain $t_i\neq0$.
\end{proof}
Let $\Gamma_0$ be the set formed by the elements $(i,j,k)\in\Gamma$ such that
\begin{itemize}
 \item $n|i-k$, and
 \item $j=0$ or $t_i=0$ or $u^iw_1^jd^k \neq z_{i,j,k}$.
\end{itemize}
\begin{lemma}
\label{lemma:ind_lineal_gamma0}
 The set consisting of the classes in $HH_0(A)$ of the elements $u^iw_1^jd^k$ with 
$(i,j,k)\in\Gamma_0$ is linearly independent.
\end{lemma}
\begin{proof}
Let $\Gamma' \subseteq \Gamma_0$ be a finite set and let $\lambda_\gamma\in\KK^\times$, 
with $\gamma\in\Gamma'$, be such that $\sum_{\gamma\in\Gamma'}\lambda_\gamma 
u^{\gamma_1}w_1^{\gamma_2}d^{\gamma_3}\in\im(d_1)$. We may further assume, without loss 
of generality, that $\sum_{\gamma\in\Gamma'}\lambda_\gamma 
u^{\gamma_1}w_1^{\gamma_2}d^{\gamma_3}$ belongs to the subspace of $\im(d_1)$ spanned by 
the homogeneous elements of special degree divisible by $n$.

It is easy to check that $f_{i,j,k}=g_{i,j,k}$ for all 
$i\geq0,j\geq1$ and $k\geq0$ such that $n|i-k$. Therefore, the subspace of 
$\im(d_1)$ spanned by the homogeneous elements of special degree divisible by $n$ is the 
$\KK$-span of the set 
\[\{f_{i,j,k}:i\geq-1,j\geq1,k\geq0\}\cup\{g_{i,j,-1}:i\geq0,j\geq1\}.\] Thus,
\begin{equation}
\label{eq:ind_lineal_gamma0}
 \sum_{\gamma\in\Gamma'}\lambda_\gamma 
u^{\gamma_1}w_1^{\gamma_2}d^{\gamma_3}=\sum_{i\geq0,j\geq1,k\geq0}\mu_{i,j,k}f_{i,j,k
} + \sum_{j\geq1,k\geq0}\mu_{j,k}f_{-1,j,k} + \sum_{i\geq0,j\geq1}\mu_{i,j}'g_{i,j,-1}.
\end{equation}
Let $(a,b,c)$ be an element in 
$\Gamma'$ and denote $L=L_{a,b}$. By Lemma \ref{lemma:hh0_1} we obtain $b\geq1$ and 
$t_a\neq0$. As a consequence $u^aw_1^bd^c\neq z_{a,b,c}$. This implies $s_{a+m,b-m}\neq0$ 
for all $m=1,\dots,L+1$.

Notice that 
$f_{-1,j,k}=(1-r_1^{j-1})w_1^{j-1}d^{k+1}$ and 
$g_{i,j,-1}=(1-r_1^{j-1})u^{i+1}w_1^{j-1}$, and that the elements $(0,j-1,k+1)$ and 
$(i+1,j-1,0)$ belong to $\Gamma$ if and only if $1-r_1^{j-1}=0$. Since 
$(a,b,c)\in\Gamma'\subseteq\Gamma$, the coefficient of $u^aw_1^bd^c$ on the right hand 
side of the above equation is
\[
 \mu_{a-1,b+1,c-1}s_{a,b} - \mu_{a,b,c}t_a.
\]
On the left hand side its coefficient is $\lambda_{a,b,c}$. Since $(a,b,c)\in\Gamma$, it 
follows that $\mu_{a-1,b+1,c-1}s_{a,b}=0$. Therefore $\mu_{a,b,c}= 
-\lambda_{a,b,c}t_{a}^{-1}\neq0$. 
The fact that $s_{a+m,b-m}\neq0$ for $m=1,\dots,L+1$ implies 
$(a+m,b-m,c+m)\notin\Gamma$ and as a consequence the coefficient of $u^{a+m}w_1^{b-m}d^{c+m}$ 
on the left hand side of Equation \ref{eq:ind_lineal_gamma0} is $0$, for the same values 
of $m$. We thus obtain
\[\mu_{a+l,b-l,c+l} = \mu_{a+l-1,b-l+1,c+l-1}\frac{s_{a+l,b-l}}{t_{a+l}},\]

for $1\leq l\leq L$. We deduce $\mu_{a+L,b-L,c+L}\neq0$. 
On the other hand, either $L=b-1$ or $t_{a+L+1}=0$. In either case 
$\mu_{a+L+1,b-L-1,c+L+1}t_{a+L+1}=0$.
Looking at the coefficient of 
$u^{a+L+1}w_1^{b-L-1}d^{c+L+1}$ on both sides of \eqref{eq:ind_lineal_gamma0} we 
obtain $\mu_{a+L,b-L,c+L}s_{a+L+1,b-L-1}=0$. This is a contradiction.
\end{proof}

\begin{proof}[Proof of Proposition \ref{proposition:hh0}]
Using Equations \eqref{eq:1sthh0} and \eqref{eq:2ndhh0} in order to obtain rewriting 
rules, it is clear that $HH_0(A)$ is generated by the classes of the elements 
$u^iw_1^jd^k$ with $(i,j,k)\in\Gamma$.

For an algebra $A$ in the family (F1), $n=0$ and $\Gamma$ is the set
\[
 \{(0,0,k):k\geq0\}\cup\{(i,0,0):i\geq0\}\cup\{(0,j,0):j\geq0\}.
\]
Suppose 
\[
 \sum_{k\geq1}\lambda_kd^k + \sum_{i\geq1}\mu_iu^i + \sum_{j\geq0}\epsilon_jw_1^j \in 
\im(d_1)
\]
for some $\lambda_k,\mu_i,\epsilon_j\in\KK$. By Lemma \ref{lemma:hh0_1} it follows that
$\lambda_k=0=\mu_i$ for all $i,k\geq1$ and $\epsilon_0=0$. Since $A$ belongs to (F1), the 
element $s_{m,j-m}\neq0$ for all $j\geq1$ and $1\leq m\leq j$. This implies 
that $(0,j,0)\in\Gamma_0$ for all $j\geq1$. By Lemma \ref{lemma:ind_lineal_gamma0} we have
that $\epsilon_j=0$ for all $j\geq1$. As a consequence, the classes in $HH_0(A)$ of the 
elements of the set 
\[
 \{d^k:k\geq1\}\cup\{u^i:i\geq1\}\cup\{w^j:j\geq0\},
\]
form a basis and we obtain the first claim of Proposition \ref{proposition:hh0}.

Let $A$ be an algebra in the family (F2) such that $r_1$ is different from $1$ and $-1$. 
In this case $r_2=r_1^{-1}$ and $n$ is different from $1$ and $2$. Here $\Gamma$ is 
the set of elements $(i,j,k)\in\NN_0^3$ satisfying any of the following 
properties.
\begin{multicols}{2}
\begin{enumerate}
 \item $n|j-i$ and $n|j-k$.
 \label{tipo1}
 \item $i=j=0$ and $n\nmid k$.
 \label{tipo6}
 \item $j=k=0$ and $n\nmid i$.
 \label{tipo4} 
 \item $i=k=0$ and $n\nmid j$.
 \label{tipo5}
 \item $i=0$, $n|j$, $n\nmid k$ and $j\geq1$.
 \label{tipo2}
 \item $k=0$, $n|j$, $n\nmid i$ and $j\geq1$.
 \label{tipo3}
\end{enumerate}
\end{multicols}
Let us see that the elements $u^iw_1^jd^k$ with $(i,j,k)$ of types \ref*{tipo2} and 
\ref*{tipo3} belong to $\im(d_1)$. Let $(i,j,k)$ be of type \ref*{tipo2}. We have 
$w_1^jd^k = r_1(1-r_1^{-2})uw_1^{j-1}d^{k+1} - r_1f_{0,j,k}$ and 
\[
 uw_1^{j-1}d^{k+1}=\frac{\phi_kf_{0,j,k} - g_{0,j,k}}{r_1^k-1}.
\]
We deduce that the 
element $w_1^jd^k$ belongs to $\im(d_1)$. The case where $(i,j,k)$ is of type \ref*{tipo3} 
is similar. As a consequence, the homology 
space $HH_0(A)$ is generated by the classes of elements $u^iw_1^jd^k$ with 
$(i,j,k)$ of type \ref*{tipo1}, \ref*{tipo6}, \ref*{tipo4} or \ref*{tipo5}. Observe that 
if $(i,j,k)\in\Gamma$ is not of type \ref*{tipo2} or \ref*{tipo3}, then either $j=0$ or 
$n|i-k$. If $j\geq1$ and $(i,j,k)\notin\Gamma_0$, then 
$u^iw_1^jd^k=z_{i,j,k}\in\im(d_1)$. Thus, we can remove it from our set of generators.
Using Lemmas \ref{lemma:hh0_1} and \ref{lemma:ind_lineal_gamma0} we deduce that the set 
of classes of elements $u^iw_1^jd^k$ with $(i,j,k)$ in 
\[
 \Gamma_1\colonequals\{(i,j,k)\in\Gamma: (i,j,k)\text{ is of type 
}\ref*{tipo1},\ref*{tipo6},\ref*{tipo4}\text{ or }\ref*{tipo5}\text{, and } j=0 \text{ 
or } (i,j,k)\in\Gamma_0\},
\]
is a basis of $HH_0(A)$. Now we describe the set $\Gamma_1$. 

Let $(i,j,k)$ be of type \ref*{tipo1} with $j\geq1$ and $t_i\neq0$. Denote $L=L_{i,j}$. 
We have $t_{i+l}=0$ if and only if $n|2(i+l+1)$. On the other hand $s_{i+m,j-m}=0$ if and 
only if $n|2m$. In particular $s_{i,j}=0$. If $n$ is odd, then there exists 
$0\leq l\leq n-1$ such that $n|2(i+l+1)$, which implies $L\leq n-2$. 
Similarly, in case $n$ is even, we obtain $L\leq n/2-2$. In either 
case $s_{m,j-m}\neq0$ for all $m=1,\dots,L+1$ and as a consequence 
$u^iw_1^jz^k\neq z_{i,j,k}$. 
This implies $(i,j,k)\in\Gamma_1$ for all $(i,j,k)$ of type \ref*{tipo1}.

Suppose $n$ is even. Let $(0,j,0)$ be of type \ref*{tipo5}. Since 
$n\nmid j$ we have 
$j\geq1$. Let $L=L_{0,j}$. The element $t_l$ is zero if and only if $n|2(l+1)$. 
Therefore $L=\min\{n/2-2,j-1\}.$
On the other hand 
$s_{m,j-m}$ is zero if and only if $n|j-2m$. If $j$ is odd, then this last condition is 
never satisfied and $s_{m,j-m}\neq0$ for all $m$, from where we deduce $w^j\neq 
z_{0,j,0}$ and $(0,j,0)\in\Gamma_1$. Suppose $j$ is even. If $L=j-1$, then $s_{m,j-m}=0$ 
for $m=j/2$ and $j/2<L+1$, which implies $w^j=z_{0,j,0}$. If $L=n/2-2$, so there exists 
$1\leq m \leq n/2-1=L+1$ such that $n/2$ divides $j/2 - m$. This implies $s_{m,j-m}=0$ 
and $w^j=z_{0,j,0}$. We conclude that an element $(0,j,0)$ of type \ref*{tipo5} belongs 
to $\Gamma_1$ if and only if it $j$ is odd. We have proven the first part of the second 
claim of Proposition \ref{proposition:hh0}.

Suppose $n$ is odd and let $(0,j,0)$ be of type \ref*{tipo5}. Set $L=L_{0,j}$. In this 
case $t_l=0$ if and only if $n|l+1$. We deduce $L=\min\{j-1,n-2\}$. On the other hand, 
$s_{m,j-m}=0$ if and only if $n|j-2m$. Since $2$ is invertible modulo $n$ and $n\nmid j$, 
this condition is always satisfied for some $1\leq m\leq n-1$. Therefore, if $j\geq n-1$, 
we obtain $w^j=z_{0,j,0}$ and $(0,j,0)\notin\Gamma_1$. Suppose $j\leq n-2$. We have 
$L=j-1$. The absolute value of $j-2m$ is positive and strictly less than $n$ for all 
$m=1,\dots,j$. Thus, if $1\leq m\leq n-1$ is such that $n|j-2m$, we get $m\geq j+1$. This 
implies
\[
 \max\{m\geq1: s_{m,j-m}\neq0\}\geq j+1> L+1.
\]
As a consequence $w^j\neq z_{0,j,0}$ and $(0,j,0)\in\Gamma_1$. This proves the second 
part of the second claim Proposition \ref{proposition:hh0}.

The proof of the cases where $n=1$ or $n=2$ is similar.
\end{proof}

We will now describe $HH_3(A)$. The result depends heavily on whether the algebra belongs 
to (F1) or to (F2). In the first case $HH_3(A)$ annihilates, while for (F2) the dimension 
is always infinite, for which the basis differs considerably in the root of unity case.

\begin{proposition}
\label{proposition:hh3}
Let $A=A(\alpha,\beta,0)$ be a down-up algebra. 
\begin{enumerate}
 \item If $A$ belongs to (F1), then $HH_{3}(A)$ vanishes. 
 \item If $A$ belongs to (F2), the 
Hochschild homology group $HH_{3}(A)$ 
has a basis formed by the classes of the elements of the set 
\begin{itemize}
\item $\{ w_1^{2i}w_2^{2i} | d^2u^2 : i\geq0\}$ if $r_1$ is not a root of unity,
\item $\{\ w_1^iw_2^ju^{nk}d^{nl}|d^2u^2: n|i-j\text{ and }kl=0\}$ if $r_1$ is a 
primitive $n$-th root of unity with $n\geq3$,
\item $\{ w_1^{2i}|d^2u^2 : i\geq0\}$ if $r_1=-1$,
\item $\{ w_1^{i} |d^2u^2 : i\geq0\}$ if $r_1=1$.
\end{itemize}
\end{enumerate}
\end{proposition}
\begin{proof} 
Let $v\in\Ker d_3$. Since the differentials respect the bidegree, we may assume 
$v$ is homogeneuos of bidegree $(s,t)$. Let 
$l=(s+t)/2$. An element $u^iw_1^jd^k$ homogeneous of bidegree $(s,t)$ satisfies 
$j = l-i\geq0$ and $k=i-t\geq0$. As a consequence, we deduce $l\geq0$ and $t\leq 
l$. Set 
\[     v = \sum_{i} c_{i} u^{i} w_{1}^{l-i} d^{i-t},     \]
where $c_i$ vanishes either when $i<0$, $l-i<0$ or when $i-t<0$. Using the formulas 
in Fact 
\ref{fact:comm} we obtain the following equality.
\begin{equation*}
  \begin{split}
     d_{3}&(v | d^{2} u^{2})\\
     &= - \Big(\sum c_{i} \big( (1 + \beta r_{1}^{l-i} r_{2}^{i-t}) u^{i+1} 
     w_{1}^{l-i} d^{i-t} + \beta \frac{\phi_{i-t-1} }{r_{1}} u^{i} 
     w_{1}^{l-i+1} d^{i-t-1} \big) \Big) | d^{2} u\\
     &+\Big(\sum c_{i} \big( (1+ \beta r_{1}^{l-i} r_{2}^{i}) u^{i} w_{1}^{l-i} 
     d^{i-t+1} + \beta \frac{\phi_{i-1} }{r_{1}} u^{i-1} 
w_{1}^{l-i+1}d^{i-t} \big) \Big) | d u^{2}.
  \end{split}
\end{equation*}
The condition $d_{3}(v | d^{2} u^{2})=0$ implies the vanishing of each summand 
separately. By looking at the coefficient of 
$u^{a+1}w_1^{l-a}d^{a-t}$ in the first constraint and at the coefficient of 
$u^aw_1^{l-a}d^{a-t+1}$ in the second constraint, we obtain the following 
identities. For all $a\geq0$,
\begin{align*}
 0&=c_a(1+\beta r_1^{l-a}r_2^{a-t}) + c_{a+1}\beta 
\frac{\phi_{a-t} }{r_1},\\
 0&=c_a(1+\beta r_1^{l-a}r_2^{a}) + c_{a+1}\beta\frac{\phi_{a} }{r_1}.
\end{align*}

Suppose $A$ belongs to (F1).
The first equality implies that $c_a = 
\mu_ac_{a+1}$, where $\mu_a = \beta\phi_{a-t} (r_1(1+\beta 
r_1^{l-a}r_2^{a-t}))^{-1}$. Since $c_a=0$ for all $a>l$, we deduce $c_a=0$ for 
all $a$. As a consequence, $HH_3(A)=0$.

Suppose now that $A$ belongs to (F2) and $r_1$ is not a root of 
unity. Using the fact that $r_2 = r_1^{-1}$, the equalities above are
\begin{align*}
 0&=c_a(1 - r_1^{l-2a+t}) - c_{a+1}\frac{\phi_{a-t} }{r_1},\\
 0&=c_a(1 - r_1^{l-2a}) - c_{a+1}\frac{\phi_{a} }{r_1}.
\end{align*}

If $l$ is odd, the second equality and an argument similar to the case (F1) 
show that $c_a=0$ for all $a$, and so $v=0$. 
Suppose $l$ is even. We may use the second equation for $a$ ranging from 
$l$ to $l/2+1$ and the previous argument to deduce $c_a=0$ for $a\in\{ 
l/2+1,\cdots,l\}$. 
Replacing $a=l/2$ in the first equation we obtain $c_{l/2}(1-r_1^t)=0$. 
If $t\neq0$, then $c_{l/2}=0$, and the first equation for values of $a$ ranging from 
$l/2-1$ to $0$ proves that $c_a=0$ for all $a$ and therefore $v=0$. If $t=0$, then the 
same argument proves that there exists $\mu_a\in k$ such that $c_a = \mu_a c_{l/2}$ for 
all $a$. 
Observe that in this case $l$ is even and $t=0$, which implies $s=4k$ for some 
$k\in\ZZ$.
As a consequence, the homogeneous component of bidegree $(s,t)$ of $\Ker d_3$ is trivial 
if $(s,t)\neq(4k,0)$ for some $k\in\ZZ$, and in case $(s,t)=(4k,0)$, it is one 
dimensional. 
Using \eqref{eq:permws} it is easy to see that the element 
$w_1^{2k}w_2^{2k}$ belongs to the homogeneous component of bidegree 
$(4k,0)$.

The case where $A$ belongs to (F2) and $r_1$ is a root of unity follows from 
\cite{Ku01}, Lemma 2.0.1, Theorems 4.0.3 and 4.0.4, together with the fact 
that $A$ is $3$-Calabi-Yau, so $HH_3(A)[4]\cong HH^0(A)$. 
\end{proof}

Theorem \ref{thm:homology} follows from Proposition \ref{proposition:hh0}, 
Proposition \ref{proposition:hh3} and the identities in \eqref{eq:hilbert_series}.

\section{Hochschild cohomology}
\label{sec:cohomology}
As we mentioned in Section \ref{sec:homology}, if $A$ belongs to (F2), then it is 
$3$-Calabi-Yau and the dimension of the Hochschild cohomology spaces can be deduced from 
Theorem \ref{thm:homology}, since $HH^i(A)$ is isomorphic to $HH_{3-i}(A)[4]$ for all 
$i\in\{0,1,2,3\}$.
We use again the minimal resolution of $A$ as $A$-bimodule to
obtain the following complex whose homology is isomorphic to the Hochschild cohomology of 
$A$.
\begin{equation}
\label{eq:coHHA}
     0 \rightarrow A 
           \overset{d_{0}^{*}}{\rightarrow} V^{*} \otimes A
           \overset{d_{1}^{*}}{\rightarrow} R^{*} \otimes A 
           \overset{d_{2}^{*}}{\rightarrow} \Omega^{*} \otimes A  
           \rightarrow 0,     
\end{equation}
where $V^{*}$ is the dual space of $V$ spanned by the basis $\{ U, D \}$, and similarly 
for $R^{*}$ and 
$\Omega^{*}$. 
The differentials are given by
\begin{align}
\label{eq:d1*}
\nonumber d_{0}^{*}(a) &= U | (u a - a u) + D | (d a - a d),\\
d_{1}^{*}(x) &= D^2U \otimes \Delta_1(x) + DU^2\otimes\Delta_2(x),
\end{align}
where, for $x=U \otimes a + D \otimes a'$,
\begin{align*}
&\Delta_1(x)=d^{2} a + a' d u + d a' u - \alpha (d a d + a' u d + d u a') 
- \beta (a d^{2} + u a' d + u d a') - \gamma a', \\
&\Delta_2(x)=d a u + d u a + a' u^{2} - \alpha (a d u + u d a + u a' u) 
- \beta (a u d + u a d + u^{2} a') - \gamma a,
\end{align*}
and
\begin{equation}
\label{eq:d2*}
d_{2}^{*}(D^{2} U \otimes a + D U^{2} \otimes a') = D^{2} U^{2} \otimes (d a' + \beta a' d 
- a u - \beta u a).
\end{equation}
Clearly $HH^i(A)=0$ for all $i\geq4$.
From now on, assume that $A$ belongs to (F1).
Note that by defining $\mathrm{bideg}(U) = (-1,-1)$ and $\mathrm{bideg}(D)=(1,-1)$, the 
differentials of the complex \eqref{eq:coHHA} are of bidegree zero. 
We recall that $HH^0(A)$ is the center $\mathscr Z(A)$ of the algebra $A$.

\begin{proposition}
\label{prop:cohh0}
Let $A=A(\alpha,\beta,0)$ be a down-up algebra of the family (F1). The 
cohomology space $HH^{0}(A)$ is $\KK.1_{A}$.
\end{proposition}
\begin{proof} 
It is clear that $\KK\cdot 1_{A} \subseteq \mathscr{Z}(A)$. 
We shall prove the other inclusion. 
Let 
\[     a = \sum c_{ijk} u^{i} w_{1}^{j} d^{k},     \]
where the sum is indexed over all integers $i, j, k \in \NN_{\geq 0}$, the 
support is finite and $c_{ijk} \in \KK$. 
Let us suppose that $a \in \mathscr{Z}(A)$, so in particular $u a - a u = 0$. 
It suffices to prove that $c_{ijk} = 0$ for all $(i, j, k) \in \NN_{\geq 0}^{3} \setminus \{ (0,0,0) \}$. 
Using the identities \eqref{eq:permws} we get that
\begin{equation}
\label{eq:1stcohh0}     
 a d - d a = \sum c_{ijk} \Big( (1 - r_{1}^{j} r_{2}^{i}) u^{i} w_{1}^{j} d^{k+1} - 
\frac{\phi_{i-1}}{r_{1}} u^{i-1} w_{1}^{j+1} d^{k} \Big),    
\end{equation}
and 
\begin{equation}
\label{eq:2ndcohh0} 
     a u - u a = - \sum c_{ijk} \Big( (1 - r_{1}^{j} r_{2}^{k}) u^{i+1} w_{1}^{j} d^{k} - 
\frac{\phi_{k-1}}{r_{1}} u^{i} w_{1}^{j+1} d^{k-1} \Big).     
\end{equation}
Since $a \in \mathscr{Z}(A)$, both two expressions vanish.
By regarding the total coefficient of the monomial $u^{i_{0}+1}d^{k_{0}}$ on the
right hand side of \eqref{eq:2ndcohh0} we get that $c_{i_{0},0,k_{0}}=0$ if 
$k_{0} \neq 0$, since in this case $(1 - r_{2}^{k_{0}}) \neq 0$.
Analogously, since the total coefficient of the monomial $u^{i_{0}}d^{k_{0}+1}$ 
on the right hand side of \eqref{eq:1stcohh0} vanishes, we see that 
$c_{i_{0},0,k_{0}}=0$ if $i_{0} \neq 0$, since in this case $(1 - 
r_{2}^{i_{0}}) \neq 0$. As a consequence, we conclude that $c_{i,0,k} = 0$, for 
all $(i,k) \in \NN_{0}^{2} \setminus \{ (0, 0) \}$. 

The vanishing of the coefficient of the monomial 
$u^{i_{0}+1}w^{j_{0}}d^{k_{0}}$ on the right hand side of \eqref{eq:2ndcohh0}, 
for $j_{0} > 0$, implies that  
\[     c_{i_{0},j_{0},k_{0}} = - \frac{c_{i_{0}+1,j_{0}-1,k_{0}+1}}{(1 - r_{2}^{k_{0}} r_{1}^{j_{0}})}.     \]
Note that the hypothesis of genericity implies that the denominator never 
vanishes. Iterating this identity we obtain that $c_{ijk}$ is proportional to 
$c_{i+j,0,k+j}$, which vanishes if $(i,j,k) \neq (0,0,0)$, and so $c_{ijk} = 0$ 
for all $(i, j, k) \in \NN_{\geq 0}^{3} \setminus \{ (0,0,0) \}$, thus proving
the proposition. 
\end{proof}
%
\begin{proposition}
\label{prop:cohh3}
Let $A=A(\alpha,\beta,0)$ be a down-up algebra of the family (F1). The  
cohomology space $HH^{3}(A)$ is isomorphic to the $\KK$-vector space spanned by 
the classes of the elements of the set 
\[
\{ D^{2} U^{2} | w_{1}^{j} : j\geq0\text{ and } j\neq2 \}\cup\{D^{2} U^{2} | 
uw_1d\}.
\]
\end{proposition}
\begin{proof} 
Identifying the space $\Omega^*\otimes A$ with $A$, the cohomology space 
$HH^3(A)$ is $A/S$, where $S$ is the $\KK$-vector space $\{da + \beta ad + a'u 
+ \beta ua':a,a'\in A\}$. 
Denote by $\pi:A\to HH^3(A)$ the canonical projection. If $v$ is an element of 
$HH^3(A)$, denote by $\KK v$ the $\KK$-vector space spanned by it.

We shall prove that the classes of 
$\{w_{1}^{j} : j\geq2 \text{ and }j\neq2\}\cup\{uw_1d\}$ in the cohomology space 
$HH^{3}(A)$ form a basis. Let $a = u^{i} w_{1}^{j} d^{k},$ where $i, j, k \in 
\NN_{\geq 0}$. 
It is straightforward to compute 
\begin{equation}
\label{eq:1stcohh3}
     da + \beta ad =  \big(\beta + r_{1}^{j} r_{2}^{i}\big) u^{i} w_{1}^{j} d^{k+1} + \frac{\phi_{i-1}(r_{1},r_{2})}{r_{1}} u^{i-1} w_{1}^{j+1} d^{k},     
\end{equation}
and 
\begin{equation}
\label{eq:2ndcohh3}
au + \beta ua= \big(\beta + r_{1}^{j} r_{2}^{k}\big) u^{i+1} w_{1}^{j} d^{k} + \beta \frac{\phi_{k-1}(r_{1},r_{2})}{r_{1}} u^{i} w_{1}^{j+1} d^{k-1}.     
\end{equation}

The hypothesis of genericity implies that $\beta + r_1^jr_2^l=0$ if and only if 
$(i,l)=(1,1)$. Note that the first coefficient on the right hand side of
the above equations is of this form. Setting $i=0$ in \eqref{eq:1stcohh3} and $k=0$ in 
\eqref{eq:2ndcohh3} we obtain that $w_{1}^{j} d^{l}$ and $u^{l}w_{1}^j$ belong 
to $S$ for all $j\geq0$ and $l\geq1$. Let $i,j,k\geq0$. Equation 
\eqref{eq:1stcohh3} implies that $\pi(u^iw_1^jd^k)\in 
\KK\pi(u^{i-1}w_1^{j+1}d^{k-1})$ for all $k\geq1$ and
$(i,j)\neq(1,1)$. 
Suppose $(i,j)\notin\{(1,1),(2,0)\}$. By a repeated use of 
\eqref{eq:1stcohh3} and the remarks above, we conclude that $u^{i} w_{1}^{j} 
d^{k}$ lies in $S$ if $i\neq k$, and we also deduce that $\pi(u^iw_1^jd^i)\in 
K\pi(w_1^{i+j})$. 
By a similar argument using \eqref{eq:2ndcohh3} we 
obtain that in case $(i,j)\in\{(1,1),(2,0)\}$, the element 
$\pi(u^iw_1^jd^k)$ belongs to $\KK\pi(uw_1d)$ for all $k$. 
As a consequence, the set $\{\pi(w_{1}^{j}) : j\geq2\}\cup\{\pi(u^2d^2)\}$ generates $HH^3(A)$ as a 
$\KK$-vector space. 
On the other hand, \eqref{eq:1stcohh3} tells us that the element $\pi(w_1^2)$ vanishes for $(i,j,k)=(1,1,0)$.

Let us see that the set $\{\pi(w_{1}^{j}) : j\geq2\}\cup\{\pi(u^2d^2)\}$ is
linearly independent. Suppose there exist elements $\lambda_j\in \KK$, with 
$j\geq0$, and $a,a'\in A$, such that 
\[
 \sum_{j\neq2}\lambda_jw_1^j + \lambda_2uw_1d = da + \beta ad + au + \beta ua.
\]
Let $l\geq0$. By looking at the homogeneous component of bidegree $(2l,0)$ in 
the equation above, we deduce that there exist $\epsilon_{k,j}\in \KK$ such that
\begin{align*}
 \lambda_l w_1^l &= \sum_{j+k=l-1}\epsilon_{j,k}(f_{j,k}),& &\text{ if }l\neq2,\\
 \lambda_2 uw_1d &= \sum_{j+k=1}\epsilon_{j,k}(f_{j,k}),& &\text{ if }l=2,
\end{align*}
where $f_{j,k} = (u^{k+1}w_1^jd^k)d + \beta d(u^{k+1}w_1^jd^k)$ for all 
$k,j\geq0$. It is easy to see that these equations imply $\lambda_j=0$ for all
$j\geq0$.
\end{proof}
%
\begin{proposition}
\label{prop:cohh1}
Let $A=A(\alpha,\beta,0)$ be a down-up algebra of the family (F1). The 
cohomology space $HH^{1}(A)$ is $2$-dimensional and it is spanned by the 
classes of $\{ D | d, U | u \}$. 
\end{proposition}
We shall first prove the following intermediate result.
\begin{lemma}
Under the same assumptions of the proposition, the $\KK$-vector space  
$(V^{*} \otimes A)/\operatorname{Im}(d^{*}_{0})$ is spanned by the classes of 
the elements of the set
\[     \mathscr{S} =  \{ U | u^{l} , D | d^{l} : l \in \NN \} \cup \{ U | u^{i} w_{1}^{j} d^{k} : i - k \leq 0 \} \cup \{ D | u^{i} w_{1}^{j} d^{k} : i - k \geq -1 \}.     \] 
\end{lemma}
\begin{proof}
Note that there is some redundancy in our description of $\mathscr{S}$, since 
for example $D|d$ belongs both to the first and to the third subset of the union. 
If $a = u^{i} w_{1}^{j} d^{k}$, then $d_{0}^{*}(a)$ equals  
\begin{multline}
\label{eq:1stcohh1} 
     U | \Big( (1 - r_{1}^{j} r_{2}^{k}) u^{i+1} w_{1}^{j} d^{k} - \frac{\phi_{k-1}(r_{1},r_{2})}{r_{1}} u^{i} w_{1}^{j+1} d^{k-1} \Big) \\
     - D | \Big( (1 - r_{1}^{j} r_{2}^{i}) u^{i} w_{1}^{j} d^{k+1} - \frac{\phi_{i-1}(r_{1},r_{2})}{r_{1}} u^{i-1} w_{1}^{j+1} d^{k} \Big).     
\end{multline}
Suppose $x_{a,b,c} = U | u^{a} w_{1}^{b} d^{c} \notin \mathscr{S}$, 
then $b + c > 0$ and $a - c > 0$. 
We shall show that $x_{a,b,c}$ belongs to the subspace spanned by $\mathscr{S}$ 
and $\operatorname{Im}(d^{*}_{0})$. 
In order to do so, first notice that Equation \eqref{eq:1stcohh1} for $i=a-1$, 
$j=b$ and $k=0$ tells us that $x_{a,b,0} = U | u^{a} w_{1}^{b}$ belongs to the 
subspace spanned by $\mathscr{S}$ and $\operatorname{Im}(d^{*}_{0})$, since for 
$b>0$ the coefficient of $x_{a,b,0}$ in \eqref{eq:1stcohh1} is nonzero 
by the hypothesis of genericity. 
Moreover, Equation \eqref{eq:1stcohh1} for $i=a-1$, $j=b$ and $k=c$ tells us 
that $x_{a,b,c} = U | u^{a} w_{1}^{b} d^{c}$ belongs to the subspace spanned by 
$x_{a-1,b+1,c-1}$, the set $\mathscr{S}$ and $\operatorname{Im}(d^{*}_{0})$, 
because, for $b+c>0$, the coefficient of $x_{a,b,c}$ in 
\eqref{eq:1stcohh1} is nonzero due to the hypothesis of genericity.
By a recursive argument we prove that $x_{a,b,c} = U | u^{a} w_{1}^{b} d^{c}$ 
belongs to the subspace spanned by $\mathscr{S}$ and 
$\operatorname{Im}(d^{*}_{0})$.

Analogously, let $x'_{a',b',c'} = D | u^{a'} w_{1}^{b'} d^{c'} \notin \mathscr{S}$. 
Thus, $a' + b' > 0$ and $a' - c' < -1$. 
We claim that $x'_{a',b',c'}$ belongs to the subspace spanned by the set 
$\mathscr{S}$ and $\operatorname{Im}(d^{*}_{0})$. 
Indeed, first notice that Equation \eqref{eq:1stcohh1} for $i=0$, $j=b'$ and 
$k=c'-1$ implies that 
$x'_{0,b',c'} = D | w_{1}^{b'} d^{c'} $ belongs to the subspace spanned by 
$\mathscr{S}$ and $\operatorname{Im}(d^{*}_{0})$, since for $b>0$ the 
coefficient of $x'_{0,b',c'}$ in \eqref{eq:1stcohh1} is nonzero by the 
hypothesis of genericity. 
Furthermore, Equation \eqref{eq:1stcohh1} for $i=a'$, $j=b'$ and $k=c'-1$ 
implies that 
$x'_{a',b',c'} = D | u^{a'} w_{1}^{b'} d^{c'}$ belongs to the subspace spanned 
by $x'_{a'-1,b'+1,c'-1}$, the set $\mathscr{S}$ and 
$\operatorname{Im}(d^{*}_{0})$, using that for $b'+c'>0$ the coefficient of 
$x'_{a',b',c'}$ in \eqref{eq:1stcohh1} is nonzero by the hypothesis of 
genericity, for $b' + c' > 0$. 
A recursive argument allows us to conclude that $x'_{a',b',c'} = D | u^{a} 
w_{1}^{b} d^{c'}$ belongs to the subspace spanned by $\mathscr{S}$ and 
$\operatorname{Im}(d^{*}_{0})$.
\end{proof} 
Since \eqref{eq:coHHA} is a complex, the differential $d_{1}^{*}$ trivially 
induces a map $\bar{d}_{1}^{*}$ from $(V^{*} \otimes 
A)/\operatorname{Im}(d^{*}_{0})$ to $R^{*} \otimes A$, whose kernel is the 
Hochschild cohomology space $HH^{1}(A)$. 
It is easy to prove that the classes of $U | u$ and $D | d$ belong to the 
kernel of $\bar{d}_{1}^{*}$, 
and that they are linearly independent, since the intersection between the 
$\KK$-vector subspace of $V^{*} \otimes A$ 
spanned by $U | u$ and $D | d$ and  $\operatorname{Im}(d^{*}_{0})$ is trivial, 
by degree reasons. In order to complete the proof of Proposition 
\ref{prop:cohh1} it suffices thus to prove the following result. 
\begin{lemma}
Assume $A$ is a down-up algebra of the family (F1).
Define $W$ to be the $\KK$-vector subspace of $(V^{*} \otimes A)/\operatorname{Im}(d^{*}_{0})$ 
spanned by the classes of the elements of the family $\mathscr{S'}$ given by
\[      \{ U | u^{l}, D | d^{l} : l \in \NN_{\geq 2} \} \cup \{ U | u^{i} w_{1}^{j} d^{k} : i - k \leq 0 \} \cup \{ D | u^{i} w_{1}^{j} d^{k} : i - k \geq -1 \} \setminus \{ D | d \}.     \] 
The intersection $W \cap \operatorname{Ker}(\bar{d}_{1}^{*})$ is trivial. 
\end{lemma}
\begin{proof}
Let $x$ be an element of $(V^{*} \otimes A)$ given by a finite linear combination of the form 
\[     x = \underset{x_{U}}{\underbrace{\sum_{i-k \leq 0} c_{i,j,k} U | u^{i} w_{1}^{j} d^{k}}} + \underset{x_{D}}{\underbrace{\sum_{i'-k' \geq -1} c'_{i',j',k'} D | u^{i'} w_{1}^{j'} d^{k'}}} + \underset{x'_{U}}{\underbrace{\sum_{l \geq 2} a_{l} U | u^{l}}} + \underset{x'_{D}}{\underbrace{\sum_{l' \geq 2} a'_{l'} D | d^{l'}}},  \]
where we exclude the case $(i',j',k') = (0,0,1)$ in the second sum. 
Since the image under $\bar{d}_{1}^{*}$ of the class of $x$ in $(V^{*} \otimes A)/\operatorname{Im}(d^{*}_{0})$ 
coincides with the image under $d_{1}^{*}$ of $x$, it suffices to prove that 
the vanishing of this last image implies that that the class of $x$ in $(V^{*} 
\otimes A)/\operatorname{Im}(d^{*}_{0})$ vanishes.
Without loss of generality we may take $x$ homogeneous for the bigrading, since 
$d_{1}^{*}$ is homogeneous of bidegree zero. 
Being homogeneous for the special degree implies that either $x = x_{U}+x'_{D}$ 
or $x = x_{D}+x'_{U}$, while being homogeneous for the usual degree restricts 
$x$ to one of the following cases: 
\begin{itemize}
\item[(i)] $x = x_{U}$ such that $\mathrm{deg}(x_{U}) + \mathrm{s\text{-}deg}(x_{U}) \neq 0$;
\item[(ii)] $x = c_{0,1,k-1} U | w d^{k-1} + c_{1,0,k} U | u d^{k} + a'_{k+1} D | d^{k+1}$, for $k \geq 1$; 
\item[(iii)] $x = x_{D}$ such that $\mathrm{deg}(x_{D}) \neq \mathrm{s\text{-}deg}(x_{D})$;
\item[(iv)] $x = c'_{i-1,1,0} D | u^{i-1} w + c'_{i,0,1} D | u^{i} d + a_{i+1} U | u^{i+1}$, for $i \geq 1$. 
\end{itemize}
Let us first consider case (iii). 
By definition of $d_{1}^{*}$, we write 
$d_{1}^{*}(x) = D^{2}U | \Delta_{1}(x) + DU^{2} | \Delta_{2}(x)$. 
An explicit computation using formulas given in Fact \ref{fact:comm} leads to 
\begin{equation*}
\begin{split}
     \Delta_{2}(x) 
     = \sum_{i'-k' \geq -1} c_{i',j',k'} &\Big( \frac{\phi_{k'-1}(r_{1},r_{2}) \phi_{k'-2}(r_{1},r_{2})}{r_{1}^{2}} u^{i'} w_{1}^{j'+2} d^{k'-2} 
     \\
     &+ \frac{\alpha\phi_{k'-1}(r_{1},r_{2})}{r_1}(r_{1}^{j'} r_{2}^{k'-1} - 1) u^{i'+1} w_{1}^{j'+1} d^{k'-1}
     \\
      &+ (r_{1}^{2 j'} r_{2}^{2 k'} - \alpha r_{1}^{j'} r_{2}^{k'} - \beta) u^{i'+2} w_{1}^{j'} d^{k'} \Big),     
\end{split}
\end{equation*}
and the coefficient of the monomial $u^{a} w_{1}^{b} d^{c}$ (where 
$a,b,c \geq 0$) is thus
\begin{multline*}
     \frac{\phi_{c+1}(r_{1},r_{2}) \phi_{c}(r_{1},r_{2})}{r_{1}^{2}} c_{a,b-2,c+2}  
     \\+ \frac{\alpha \phi_{c}(r_{1},r_{2})}{r_1} (r_{1}^{b-1} r_{2}^{c} - 1) c_{a-1,b-1,c+1} +  (r_{1}^{2 b} r_{2}^{2 c} - \alpha r_{1}^{b} r_{2}^{c} - \beta) c_{a-2,b,c}.      
\end{multline*}
The fact that $x$ belongs to the kernel of $d_{1}^{*}$ implies the vanishing of 
the previous expression. 
In particular, we see that 
\begin{align*}
&c_{a,0,c} = 0\text{ for all }c \neq 1,\\
&c_{a,1,c} = 0\text{ for all }c \neq 0,\text{ and }\\
&c_{a,b,c} = 0\text{ for all }b \geq 2.
\end{align*}
Condition $\mathrm{deg}(x_{D}) \neq \mathrm{s\text{-}deg}(x_{D})$ implies that $c_{a,0,1} 
= 0$ 
and $c_{a,1,0} = 0$ for all $a$. 
As a consequence, $x_{D}$ vanishes in $(V^{*} \otimes A)$. Case (i) is 
handled \textit{mutatis mutandi}. 

Let us now treat case (iv), where 
\[     x = c'_{i-1,1,0} D | u^{i-1} w + c'_{i,0,1} D | u^{i} d + a_{i+1} U | u^{i+1},     \] 
for some $i \geq 1$. 
We write again$d_{1}^{*}(x) = D^{2}U | \Delta_{1}(x) + DU^{2} | \Delta_{2}(x)$. 
Using the computations of the previous paragraph we see that $\Delta_{2}(D | 
u^{i-1} w)$ and $\Delta_{2}(D | u^{i} d)$ vanish. 
The expression of $d_{1}^{*}$ in \eqref{eq:d1*} together with the formulas of 
Fact \ref{fact:comm} tell us that $\Delta_{2}(U | u^{i+1})$ is given by 
\[     \frac{(r_{1}^{i+1} + r_{2}^{i+1}-\alpha)}{r_{1}} u^{i + 1} w + r_{2} (r_{1} - r_{2}) (1 - r_{2}^{i}) u^{i+2} d.     \]
Since the second coefficient of the previous expression is nonzero by the 
hypothesis of genericity, we see that 
the vanishing of $\Delta_{2}(x)$ implies that $a_{i+1}$ is zero, which we 
will assume from now on.  
We shall now turn our attention to $\Delta_{1}(x)$, for $x = c'_{i-1,1,0} D | 
u^{i-1} w + c'_{i,0,1} D | u^{i} d$. 
Using again the formulas of Fact \ref{fact:comm}, we see that 
\[     \Delta_{1}(D | u^{i-1} w) 
= \frac{1-r_{2}^{i}}{r_{1}} \left( u^{i-1} w^{2} + r_{1}^{2} (r_{2}-r_{1}) u^{i} w d \right),     \]
and 
\[     \Delta_{1}(D | u^{i} d) = \frac{\phi_{i-1}(r_{1},r_{2})}{r_{1}^{2}} \left( u^{i-1} w^{2} + r_{1}^{2} (r_{2}-r_{1}) u^{i} w d \right).     \]
The hypothesis of genericity implies that all the coefficients 
appearing in both of the previous expressions are nonzero. 
We note however that $\Delta_{1}(D | u^{i-1} w)$ and $\Delta_{1}(D | u^{i} d)$ 
are not linearly independent, and so $x$ is a multiple of $\phi_{i-1} D | 
u^{i-1} w - r_{1} (1-r_{2}^{i}) D | u^{i} d$. 
Since $\phi_{i-1} D | u^{i-1} w - r_{1} (1-r_{2}^{i}) D | u^{i} d$ is also a multiple of 
$d_{0}^{*}(u^{i})$ (see \eqref{eq:1stcohh1}), we conclude that the class of $x$ in $(V^{*} 
\otimes A)/\operatorname{Im}(d^{*}_{0})$ vanishes.  
Case (ii) is analogous.
\end{proof}

We now turn to a characterization of the space $HH^2(A(\alpha,\beta,0))$. 
As before, let $A$ denote the algebra $A(\alpha,\beta,0)$ and let $A(t,s)$ denote the 
Hilbert series of $A$ regarded as a bigraded algebra. This bigrading on $A$ 
induces a bigrading on its Hochschild cohomology, whose associated Hilbert 
series will be denoted by $HH^i(A)(t,s)$ for all $i\geq0$.

\begin{proposition}
\label{prop:serie_hilbert_HH2}
 Under the previous assumptions,
 \begin{equation*}
  HH^2(A)(t,s) = \frac{1}{t^2} + 2 + \frac{t^2}{1-t^2}.
 \end{equation*}
\end{proposition}

\begin{proof}
 Let $C^\bullet$ be the complex \eqref{eq:coHHA}. 
 Recall that the homology of $C^\bullet$
 is $HH^\bullet(A)$. Regarding 
$HH^\bullet(A)$ as a complex with zero differentials, the Euler-Poincar\'e 
characteristic $\chi_{C^{\bullet}}(t,s)$ associated to $C^\bullet$ is equal to the 
Euler-Poincar\'e characteristic 
$\chi_{HH^\bullet(A)}(t,s)$ associated to $HH^\bullet(A)$. Using the 
descriptions we obtained of $C^\bullet$ and of the Hochschild cohomology spaces 
$HH^0(A)$, $HH^1(A)$ and $HH^3(A)$, the following equalities are 
straightforward to check,
\begin{align*}
&\chi_{C^{\bullet}}(t,s) = -t^{-4}, \\
&HH^0(A)(t,s) = 1,\\
&HH^1(A)(t,s) = 2,\\
&HH^3(A)(t,s) = \frac{1}{t^4(1-t^2)}.
\end{align*}
Therefore, the equality $\chi_{C^{\bullet}}(t,s) = \chi_{HH_\bullet(A)}(t,s)$ 
is 
\[
 -\frac{1}{t^4} = 1 - 2 + HH^2(A)(t,s) - \frac{1}{t^4(1-t^2)},
\]
and the lemma follows.
\end{proof}
From the previous lemma we deduce that every homogeneous component of $HH^2(A)$ 
of bidegree different from $(2k,0)$ for $k\geq-1$ is zero. 

The following set is a $\KK$-basis of the homogeneous component of $R^*\otimes 
A$ of bidegree $(2k,0)$
\[
\{D^2U|u^aw^{k+1-a}d^{a+1} : 0\leq a\leq k+1\}
\cup
\{DU^2|u^{a+1}w^{k+1-a}d^a : 0\leq a\leq k+1\}.
\]

\begin{proposition}
\label{prop:primeros_grados}
 The homogeneous component of $HH^2(A)$ of bidegree $(-2,0)$ is 
isomorphic to the $\KK$-vector space spanned by the class of the element 
$D^2U|d + DU^2|u$. 
On the other hand, the homogeneous component of $HH^2(A)$ of bidegree $(0,0)$ 
is isomorphic to the $\KK$-vector space spanned by the classes of the elements 
$D^2U|wd + DU^2|uw$ and $D^2U|ud^2 + DU^2|u^2d$.
\end{proposition}

\begin{proof}
There are no homogeneous elements of bidegree $(-2,0)$ in $V^*\otimes A$ and 
the homogeneous component of bidegree $(-2,0)$ in $R^*\otimes A$ is spanned by 
the elements $D^2U|d$ and $DU^2|u$. The first claim follows from the 
fact that $d_2^*(\lambda D^2U|d+\mu DU^2|u)=(\mu-\lambda)D^2U^2|(du + \beta 
ud)$.

The elements of $V^*\otimes A$ of bidegree $(0,0)$ are spanned by $D|d$ and 
$U|u$. We have already seen in Proposition \ref{prop:cohh1} these elements are 
in the kernel of $d_1^*$. 
On the other hand, the elements $D^2U|wd + DU^2|uw$ and $DU^2|ud^2 + DU^2|u^2d$ 
belong to the kernel of $d_2^*$ and they are linearly independent. By Proposition 
\ref{prop:serie_hilbert_HH2} their classes form a basis of the homogeneous 
component of $HH^2(A)$ of bidegree $(0,0)$.
\end{proof}

We now turn to a description of the homogeneous components of bidegree $(2k,0)$ 
with $k\geq1$. 
For all non negative integers $x,y$ and $z$, define 
\begin{align*}
&g_{x,y,z}=d_1^*(U|u^xw_1^yd^z),\\
&f_{x,y,z}=d_1^*(D|u^x w_1^y d^z).
\end{align*}

Observe that 
$\mathrm{bideg}(g_{x,y,z}) = (x+2y+z-1,x-z-1)$ 
and $\mathrm{bideg}(f_{x,y,z}) = (x+2y+z-1,x-z+1)$. 

\begin{lemma}
\label{lemma:f_igual_g}
Let $k\geq1$. The elements $f_{z,k-z,z+1}$, where $0\leq z\leq 
k$, generate the homogeneous component of $\im(d_1^*)$ of bidegree $(2k,0)$ as 
$\KK$-vector space.
\end{lemma}

\begin{proof}
The set $\{U|u^{z+1}w_1^{k-z}d^z:0\leq z\leq k\}\cup
\{D|u^zw_1^{k-z}d^{z+1}:0\leq z\leq k\}$ generates the homogeneous component of 
bidegree $(2k,0)$ of $V^*\otimes A$ and therefore the set
$\{g_{z+1,k-z,z}:0\leq z\leq k\}\cup\{ f_{z,k-z,z+1}: 0\leq z\leq k\}$ 
generates the homogeneous component of $\im(d_1^*)$ of the same bidegree. 

For $0\leq z\leq k$, define $x_z= U|u^{z+1}w_1^{k-z}d^z - 
D|u^zw_1^{k-z}d^{z+1}$. Observe that $d_1^*(x_z) = g_{z+1,k-z,z} - 
f_{z,k-z,z+1}$. 
Using the expression given in \eqref{eq:1stcohh1} for the image of an element under $d_{0}^{*}$ we see that
\begin{equation}
 d_0^*(u^zw_1^{k-z}d^z)= (1-r_1^{k-z}r_2^z)x_z - 
 \frac{\phi_{z-1}}{r_1}x_{z-1}.
\end{equation}
Since $k\geq1$, both coefficients appearing in the above equation are non zero. 
Taking $z=0$ and applying $d_1^*$, we deduce $d_1^*(x_0)=0$ and by an 
inductive argument we obtain $d_1^*(x_z)=0$ for all $0\leq z\leq k$. As a 
consequence, $g_{z+1,k-z,z} = f_{z,k-z,z+1}$ for all $0\leq z\leq k$ and the 
result follows.
\end{proof}

Let $n$ and $m$ be the dimensions of the components of bidegree $(2k,0)$ of 
$\Ker(d_2^*)$ and $\im(d_1^*)$, respectively. It is straightforward to check 
that 
\[
 d_2^*(D^2U|u^aw_1^{k+1-a}d^{a+1} + DU^2|u^{a+1}w_1^{k+1-a}d^a)=0,
\]
for all $0\leq a\leq k+1$. We deduce that $n\geq k+2$. On the other hand, we 
know that $m\leq k+1$ by Lemma \ref{lemma:f_igual_g}. Also, Proposition 
\ref{prop:serie_hilbert_HH2} implies $n-m=1$ and it follows that $n=k+2$ and 
$m=k+1$. As a consequence, the homogeneous component of $\Ker(d_2^*)$ of 
bidegree $(2k,0)$ is the $\KK$-vector space spanned by the elements
\[
 D^2U|u^aw_1^{k+1-a}d^{a+1} + DU^2|u^{a+1}w_1^{k+1-a}d^a,
\]
for $0\leq a\leq k+1$.

Let us prove that the element $D^2U|w_1^{k+1}d + DU^2|uw_1^{k+1}$ does 
not belong to the image $\im(d_1^*)$.
By definition, $f_{z,k-z,z+1}$ is equal to 
$d_1^*(D|u^zw_1^{k-z}d^{z+1})$. We write $f_{z,k-z,z+1} = 
D^2U|\Delta_1(D|u^zw_1^{k-z}d^{z+1}) + 
DU^2|\Delta_2(D|u^zw_1^{k-z}d^{z+1})$. For $0\leq z\leq k$, Fact \ref{fact:comm} 
implies that
\begin{equation}
 \begin{split}
\label{eq:formula_f}
 \Delta_2(D|u^zw_1^{k-z}d^{z+1}) =& 
 \frac{\phi_{z}\phi_{z-1}}{r_1^2}u^zw_1^{k-z+2}d^{z-1} + 
 \frac{(r_1^{k-z}r_2^z-1)\alpha\phi_{z}}{r_1}u^{z+1}w_1^{k-z+1}d^z \\
 &+(r_1^{k-z}r_2^{z+1}-r_1)(r_1^{k-z}r_2^{z+1}-r_2)u^{z+2}w_1^{k-z}d^{z+1}.
 \end{split}
\end{equation}
Once more, the hypothesis of genericity implies that none of the coefficients 
appearing on the right hand side of the above equation annihilates.

Suppose $D^2U|w_1^{k+1}d + DU^2|uw_1^{k+1}$ belongs to $\im(d_1^*)$. By Lemma 
\ref{lemma:f_igual_g}, there 
exist $\lambda_0,\dots,\lambda_k\in \KK$ such that 
\[
 D^2U|w_1^{k+1}d + DU^2|uw_1^{k+1} = \sum_{z=0}^k\lambda_zf_{z,k-z,z+1}.
\]
Therefore, $uw_1^{k+1} = \sum_{z=0}^k 
\lambda_z\Delta_2(D|u^zw_1^{k-z}d^{z+1})$. Looking at 
the coefficient of $u^{k+2}d^{k+1}$ on the right hand side of the last 
equality, we deduce $\lambda_k=0$. Inductively in 
$z$ and looking at the coefficient of $u^{z+2}w_1^{k-z}d^{z+1}$, we deduce 
$\lambda_z=0$ for all $0\leq z\leq k$, which is a contradiction. 
Thus, the element $D^2U|w_1^{k+1}d + DU^2|uw_1^{k+1}$ does not belong to 
$\im(d_1^*)$. 
We have proven the following.

\begin{proposition}
\label{prop:cohh2}
 Let $k\geq 1$. The homogeneous component of $\Ker(d_2^*)$ of bidegree $(2k,0)$ 
is the $\KK$-vector space of dimension $k+2$ spanned by the set
\[
 \{D^2U|u^aw_1^{k+1}d^{a+1} + DU^2|u^{a+1}w_1^{k+1}d^a: 0\leq a\leq k+1\},
\]
 and the class in $HH^2(A)$ of the element $D^2U|w_1^{k+1}d + DU^2|uw_1^{k+1}$ 
generates the homogeneous component of bidegree $(2k,0)$.
\end{proposition}

Theorem \ref{thm:cohomology} follows from Propositions \ref{prop:cohh0},\ref{prop:cohh3}, \ref{prop:cohh1}, \ref{prop:serie_hilbert_HH2}, \ref{prop:primeros_grados} and \ref{prop:cohh2}.

\bibliographystyle{model1-num-names}

\begin{bibdiv}
\begin{biblist}
\bib{Ba97}{article}{
   author={Bardzell, Michael J.},
   title={The alternating syzygy behavior of monomial algebras},
   journal={J. Algebra},
   volume={188},
   date={1997},
   number={1},
   pages={69--89},
}

\bib{B}{article}{
   author={Benkart, Georgia},
   title={Down-up algebras and Witten's deformations of the universal
   enveloping algebra of $\germ s\germ l_2$},
   conference={
      title={Recent progress in algebra},
      address={Taejon/Seoul},
      date={1997},
   },
   book={
      series={Contemp. Math.},
      volume={224},
      publisher={Amer. Math. Soc., Providence, RI},
   },
   date={1999},
   pages={29--45},
}

\bib{BR98}{article}{
   author={Benkart, Georgia},
   author={Roby, Tom},
   title={Down-up algebras},
   journal={J. Algebra},
   volume={209},
   date={1998},
   number={1},
   pages={305--344},
} 

\bib{BW01}{article}{
   author={Benkart, Georgia},
   author={Witherspoon, Sarah},
   title={A Hopf structure for down-up algebras},
   journal={Math. Z.},
   volume={238},
   date={2001},
   number={3},
   pages={523--553},
}

\bib{CL09}{article}{
   author={Carvalho, Paula A. A. B.},
   author={Lopes, Samuel A.},
   title={Automorphisms of generalized down-up algebras},
   journal={Comm. Algebra},
   volume={37},
   date={2009},
   number={5},
   pages={1622--1646},
}

\bib{CM00}{article}{
   author={Carvalho, Paula A. A. B.},
   author={Musson, Ian M.},
   title={Down-up algebras and their representation theory},
   journal={J. Algebra},
   volume={228},
   date={2000},
   number={1},
   pages={286--310},
}

\bib{CS04}{article}{
   author={Cassidy, Thomas},
   author={Shelton, Brad},
   title={Basic properties of generalized down-up algebras},
   journal={J. Algebra},
   volume={279},
   date={2004},
   number={1},
   pages={402--421},
}

\bib{CS15}{article}{
   author={Chouhy, Sergio},
   author={Solotar, Andrea},
   title={Projective resolutions of associative algebras and ambiguities},
   journal={J. Algebra},
   volume={432},
   date={2015},
   pages={22--61},
}

\bib{Fo86}{article}{
   author={Fomin, S. V.},
   title={The generalized Robinson-Schensted-Knuth correspondence},
   language={Russian},
   journal={Zap. Nauchn. Sem. Leningrad. Otdel. Mat. Inst. Steklov.
   (LOMI)},
   volume={155},
   date={1986},
   number={Differentsialnaya Geometriya, Gruppy Li i Mekh. VIII},
   pages={156--175, 195},
   translation={
      journal={J. Soviet Math.},
      volume={41},
      date={1988},
      number={2},
      pages={979--991},
      issn={0090-4104},
   },
}

\bib{Go85}{article}{
   author={Goodwillie, Thomas},
   title={Cyclic homology, derivations and the free loop spaceBasic properties of generalized down-up algebras},
   journal={Topology},
   volume={24},
   date={2004},
   number={2},
   pages={187--215},
}

\bib{HS12}{article}{
   author={Herscovich, Estanislao},
   author={Solotar, Andrea},
   title={Hochschild and cyclic homology of Yang-Mills algebras},
   journal={J. Reine Angew. Math.},
   volume={665},
   date={2012},
   pages={73--156},
}

\bib{Ig92}{article}{
   author={Igusa, K.},
   title={Cyclic homology and the determinant of the Cartan matrix},
   journal={J. Pure Appl. Algebra},
   volume={83},
   date={1992},
   pages={101--119},
}

\bib{KMP99}{article}{
   author={Kirkman, Ellen},
   author={Musson, Ian M.},
   author={Passman, D. S.},
   title={Noetherian down-up algebras},
   journal={Proc. Amer. Math. Soc.},
   volume={127},
   date={1999},
   number={11},
   pages={3161--3167},
}

\bib{Ku01}{article}{
   author={Kulkarni, Rajesh S.},
   title={Down-up algebras and their representations},
   journal={J. Algebra},
   volume={245},
   date={2001},
   number={2},
   pages={431--462},
}


\bib{SL16}{article}{
   author={Shen, Yuan},
   author={Lu, DiMing},
   title={Nakayama automorphisms of PBW deformations and Hopf actions},
   journal={Sci. China Math.},
   volume={59},
   date={2016},
   number={4},
   pages={661--672},
}

\bib{Sta88}{article}{
   author={Stanley, Richard P.},
   title={Differential posets},
   journal={J. Amer. Math. Soc.},
   volume={1},
   date={1988},
   number={4},
   pages={919--961},
}


\bib{Ta15}{article}{
   author={Tang, Xin},
   title={Algebra endomorphisms and derivations of some localized down-up
   algebras},
   journal={J. Algebra Appl.},
   volume={14},
   date={2015},
   number={3},
   pages={1550034, 14},
}

\bib{Ter90}{article}{
   author={Terwilliger, Paul},
   title={The incidence algebra of a uniform poset},
   conference={
      title={Coding theory and design theory, Part I},
   },
   book={
      series={IMA Vol. Math. Appl.},
      volume={20},
      publisher={Springer, New York},
   },
   date={1990},
   pages={193--212},
}

\bib{VP91}{article}{
   author={Vigu\'e-Poirrier, Micheline},
   title={Cyclic homology of algebraic hypersurfaces},
   journal={J. Pure Appl. Algebra},
   volume={72},
   date={1991},
   pages={95--108},
}

\bib{Zh99}{article}{
   author={Zhao, Kaiming},
   title={Centers of down-up algebras},
   journal={J. Algebra},
   volume={214},
   date={1999},
   pages={103--121},
}

\end{biblist}
\end{bibdiv}
\addcontentsline{toc}{section}{References}



\footnotesize
\noindent S.C.:
\\ IMAS, UBA-CONICET,
Consejo Nacional de Investigaciones Cient\'\i cas y T\'ecnicas, \\
Ciudad Universitaria, Pabell\'on I, 1428 Buenos Aires, Argentina
\\{\tt schouhy@dm.uba.ar}

\medskip

\noindent E.H.:
\\
Institut Fourier, Universit\'e Grenoble Alpes, \\
100 rue des Maths, 38610 Gi\`eres, France;
\\Departamento de Matem\'atica, Facultad de Ciencias Exactas y Naturales, Universidad
de Buenos Aires,\\
Ciudad Universitaria, Pabell\'on I, 1428 Buenos Aires, Argentina; and \\
IMAS, UBA-CONICET,
Consejo Nacional de Investigaciones Cient\'\i ficas y T\'ecnicas, Argentina
\\{\tt Estanislao.Herscovich@univ-grenoble-alpes.fr}

\medskip

\noindent A.S.:
\\Departamento de Matem\'atica, Facultad de Ciencias Exactas y Naturales, Universidad
de Buenos Aires,\\
Ciudad Universitaria, Pabell\'on I, 1428 Buenos Aires, Argentina; and \\
IMAS, UBA-CONICET,
Consejo Nacional de Investigaciones Cient\'\i ficas y T\'ecnicas, Argentina
\\{\tt asolotar@dm.uba.ar}

\end{document}